\documentclass[10pt]{amsart}
\usepackage{amsmath,amssymb,amsthm}
\usepackage{graphicx}
\usepackage{color}
\newtheorem{theorem}{Theorem}[section]
\newtheorem{lemma}[theorem]{Lemma} 
\newtheorem*{claim*}{Claim} 
\newtheorem{proposition}[theorem]{Proposition} 
 
\newtheorem{question}{Question} 
\newtheorem{corollary}[theorem]{Corollary}
\theoremstyle{definition}
\newtheorem{definition}[theorem]{Definition}

\newtheorem{remark}[theorem]{Remark}
\newtheorem*{remark*}{Remark}

\newcommand{\Z}{\mathbb{Z}}
\newcommand{\R}{\mathbb{R}}
\newcommand{\C}{\mathbb{C}}

\newif\ifred
\redfalse
\newcommand{\correction}{\ifred \color{red} \fi}

 \makeatletter
    
    \@addtoreset{equation}{section}
  \makeatother

\begin{document}

\title[Satellite fully positive braid links]{Satellite fully positive braid links are braided satellite of fully positive braid links} 

\author{Tetsuya Ito}
\address{Department of Mathematics, Kyoto University, Kyoto 606-8502, JAPAN}
\email{tetitoh@math.kyoto-u.ac.jp}
\
%


\begin{abstract}
A link in $S^{3}$ is a \emph{fully positive braid link} if it is the closure of a positive braid that contains at least one full-twist.
We show that a fully positive braid link is a satellite link if and only if it is the satellite of a fully positive braid link $C$ such that the pattern is a positive braid that contains sufficiently many full twists, where the number of necessary full twists only depends on $C$.
As an application, we give a characterization of the unknot by the property that certain braided satellite is a (fully) positive braid knot.
\end{abstract}

\maketitle

\section{Introduction}

An element of the $n$-strand braid group $B_{n}$ is \emph{positive} if it is a product of positive standard Artin generators $\sigma_1,\ldots,\sigma_{n-1}$.
The set of positive $n$-braid forms a submonoid $B_{n}^{+}$ called the positive braid monoid.

A link $L$ is a \emph{positive braid link} if $L$ is represented by a positive braid. It is known that a positive braid link has various nice properties. A positive braid link is fibered and its polynomial invariants, like the Conway, the Jones and the HOMFLY polynomials satisfy several interesting properties \cite{vB,St,It-HOMFLY}. Among them, it deserves to mention that these polynomials of a positive braid knot know the number of the prime factors of $K$. In particular these polynomials detect the primeness (see Proposition \ref{proposition:Alexander} in Appendix). 

Furthermore, the following class of positive braid link, which we call a \emph{fully positive braid link}\footnote{In \cite{KM}, fully positive braid knot is called \emph{twist positive knot}. To avoid the confusion with positive twist knot (a twist knot that is also a positive knot), and to emphasize the \emph{braid} must contain the \emph{full} twist (like twisted torus links, often `twist' is used to mean non-full twists), we use the terminology fully positive braid link.} has more nice properties.

Let $\Delta^{2}=(\sigma_1\sigma_2\cdots \sigma_{n-1})^{n} \in B_{n}$ be the full-twist braid. For $k \geq 0$ and a positive $n$-braid $\beta \in B_{n}^{+}$ we say that \emph{$\beta$ contains at least $k$ full twists} if $\Delta^{-2k}\beta \in B_n^{+}$. Namely, $\beta = \Delta^{2k}\alpha$ for some positive braid $\alpha \in B_{n}^{+}$.
 
\begin{definition}[Fully positive braid link]
A link $L$ is a \emph{fully positive braid link} if $L$ is represented by a positive braid that contains at least one full-twist. 
\end{definition}

In \cite{FW} it is shown that if a link $L$ is represented by a positive $n$-braid that contains at least one full twist, then the braid index $b(L)$ is equal to $n$, and that the Morton-Franks-Williams inequality of the braid index is always an equality. Similarly, in \cite{KM} it is shown that for a fully positive braid \emph{knot}, the braid index appears as the third exponent in its Alexander polynomial.
Fully positive braid links contain Lorenz links, the links that appear as periodic orbits of the Lorenz equation \cite{BW,BK}.

The aim of this paper is to give a characterization of fully positive 
braid link which are satellite.

Although we will give more precise definitions in Section \ref{Section:satellite}, to state our result we informally review our terminologies of satellites. The \emph{satellite} with a \emph{companion} $C = C_1 \cup C_2 \cup \cdots \cup C_m$ and a \emph{pattern} $P=(P_1,\ldots, P_m)$ is the link $C_P$ obtained from the oriented link $C$ by replacing each tubular neighborhood $C_i$ with the link $P_i$ in a solid torus. The satellite $C_P$ is \emph{braided} if all $P_i$ are closed braids in the solid tori. A link $L$ is a \emph{satellite link} if it is written as a satellite $C_P$. We remark that a satellite link $L$ may have many expressions as a satellite $C_P$.



As we will discuss in Section \ref{Section:satellite}, it is easy to see that if the companion $C$ is a positive braid link and the pattern $P$ is a positive braid, then $C_P$ is a positive braid link, provided $P$ contains sufficiently many full twists. We will specify how many twists are needed in Proposition \ref{proposition:regular-positive}, but in general, the number of necessary full twists depends on a choice of a braid representative of $C$.

As is well-known, the $(2,3)$-cable of the $(2,3)$-torus knot is not braid positive. Thus the satellite $C_P$ with a (fully) positive braid link companion $C$ with positive braid patterns $P$ are not necessarily a positive braid link, even if we assume that the pattern $P$ contains at least one full-twist. The pattern $P$, in general, must contain many full twists. We also also remark that as the connected sum of positive braid link shows, a non-braided satellite can yield a positive braid link.

However, we show that for a fully positive braid link, the converse holds and 
we prove the following characterization of satellite fully positive braid links.

\begin{theorem}[Characterization of satellite fully positive braid link]
\label{theorem:main}
Let $L$ be the satellite with companion $C=C_1 \cup C_2\cup \cdots \cup C_{m}$ and pattern $P=(P_1,\ldots,P_m)$. Then (a) and (b) are equivalent.
\begin{itemize}
\item[(a)] $L$ is a fully positive braid link. 
\item[(b)] $C$ is a fully positive braid link and the pattern $P$ is braided. The $i$-th pattern $P_i$ is a positive braid that contains at least 
{\correction $2g(C_i)+2b(C_i)-1$} full twists. 
\end{itemize}
Here $g(C_i)$ and $b(C_i)$ are the genus and the braid index of the knot $C_i$, respectively.
\end{theorem}

Roughly speaking, our theorem says that a fully positive braid link is a satellite if and only if it is obviously so. In particular, a satellite fully positive braid link must be a braided satellite, with a fully positive braid companion and fully positive braid patterns (i.e. the pattern is a positive braid that contains at least one full twist. 

For the knot case, the companion knot $C$ is not the unknot hence $b(C)\geq 2$ and $g(C) \geq 1$. Thus the number of full twists of a pattern $P$ needed to get a positive braid \emph{knot} $C_P$ is at least {\correction five}.
 Moreover, $b(C)=2$ and $g(C)=1$ happens if and only if $C$ is the trefoil.

Thus we get the following characterization of the unknot and trefoil. 

\begin{corollary}
\label{corollary:char-unknot}
Let $P$ be a braided pattern that contains at least one full twist, but does not contain more than 
{\correction four} full twists. Then a knot $K$ is the unknot if and only if the satellite $K_P$ is a fully positive braid knot.
\end{corollary}

\begin{corollary}
\label{corollary:char-trefoil}
Let $P$ be a braided pattern that contains at least 
{\correction five} full twists, but does not contain more than 
{\correction five} full twists. Then a non-trivial knot $K$ is the trefoil if and only if the satellite $K_P$ is a fully positive braid knot.
\end{corollary}

For example, the corollary says that the trefoil is the unique non-trivial knot whose 
{\correction $(2,11)$-cable} is a fully positive braid knot.

It has been conjectured that a satellite Lorentz knot is a cable of a Lorenz knot. Although this has been disproved in \cite{dPP}, the question is updated as follows.

\begin{question}
\label{Question:Lorenz}
Assume that a Lorenz knot is $K$ a satellite $C_P$. 
\begin{itemize}
\item[(i)] Is the companion $C$ a Lorenz knot ? \cite[Question 5]{BK}
\item[(ii)] Let us view the pattern $P$ as a knot in $S^{3}$ via the standard embedding of the solid torus. Is the pattern $P$ a Lorenz knot ? \cite[Conjecture 1.2]{dPP}
\end{itemize}
\end{question}

Since Lorenz knots are fully positive braid knots, we have the following supporting evidence for the affirmative answer to Question \ref{Question:Lorenz}.

\begin{corollary}
\label{corollary:Lorenz}
Let $K$ be a Lorenz knot. If $K$ is a satellite knot, then $K$ is the braided satellite of a fully positive braid knot and its pattern is a positive braid that contains at least 
{\correction five} full twists.
\end{corollary}

In general it is hard to distinguish a fully positive braid link with a Lorenz link, although it is easy to give an example of a fully positive braid link that is not a Lorenz link. 
Thus Corollary \ref{corollary:Lorenz} illustrates the subtlety and difficulty of Question \ref{Question:Lorenz}.

Our proof of Theorem \ref{theorem:main} is based on results coming from various different aspects of braids. We give an outline of the proof and the contents of the paper.

In Section \ref{Section:satellite} we review satellites and braided satellites.
We give a \emph{regular form}, a certain nice braid representative of a braided satellite $C_P$ and we prove the easier implication, (b) $\Rightarrow$ (a) of Theorem \ref{theorem:main}.

In Section \ref{Section:reducible} we see a braid as an element of mapping class group. Then, following the Nielsen-Thurston classification, braids are classified into three types, \emph{periodic, reducible,} and \emph{pseudo-Anosov}. 
We review how a reducible braid gives rise to a regular form\footnote{The terminology ``regular form'' comes from the regular form of reducible braids introduced in \cite{GW1}.} after taking suitable conjugate. We introduce a notion of a braid representative \emph{compatible} with the satellite $C_P$, and show that a regular form of a compatible braid gives rise to braid representatives of the pattern $C$ and the companion $P$. 

In Section \ref{Section:Garside}, we show that if a link $L$ admits a \emph{positive} braid representative compatible with the satellite $C_P$, then the companion $C$ is a positive braid link and the pattern $P$ is a positive braid, by using a regular form (Theorem \ref{theorem:compatible-satellite} -- we remark that this result holds for positive braid links, not only fully positive braid links).  
Since we need to take conjugates to get a regular form, even if we start from a positive braid, a regular form may be far from a positive braid.
We use the results from Garside theory to overcome this difficulty.

Finally, in Section \ref{Section:FDTC} we show an existence of compatible braid representatives of a satellite $C_P$ for a fully positive braid link (Theorem \ref{theorem:compatible-fully-positive}). Unlike the discussions in previous sections, here it is crucial that the braid contains at least one full twist. 
The proof is based on a result from geometric method, the braid foliation/open book foliations \cite{It,IK}. 
Theorem \ref{theorem:compatible-fully-positive} and Theorem \ref{theorem:compatible-satellite} complete the proof of the implication (a) $\Rightarrow$ (b) of Theorem \ref{theorem:main}.
We also give a similar characterization for satellite positive braid links, under some additional assumptions (Theorem \ref{theorem:main-variant}).

In Appendix we prove a characterization of the unknot that is similar to Corollary \ref{corollary:char-unknot} (Theorem \ref{theorem:var-K}) that generalizes a result in \cite{Kr}, with an elementary proof. Though the argument of Appendix is independent from the rest of the article, the result suggests a satellite positive braid link is quite special.

\section*{Acknowledgement}
The author is partially supported by JSPS KAKENHI Grant Numbers 19K03490, 21H04428, 	23K03110. The author would like to thank Thiago de Paiva for stimulating discussion for Lorenz links that inspires the author to investigate the satellite structure of positive braid links. {\correction The author also would like to thank Sangyop Lee for pointing out an error for the number of necessary twists in the earlier version of the paper.}

\section{Braided satellites and its regular form}
\label{Section:satellite}

\subsection{Satellite links}
\label{Section:satellite-definitions}

In this section we review our terminologies of satellite construction for link case.
In the following, all the objects are oriented unless otherwise specified.

Let $C = C_1\cup C_2 \cup \cdots \cup C_{m}$ be an oriented $m$-component link in $S^{3}$. We denote by $N(C)= N(C_1) \cup N(C_2)\cup \cdots \cup N(C_m)$ a (closed) tubular neighborhood of $C$, where $N(C_i)$ is a (closed) tubular neighborhood of $C_i$.

The $i$-th \emph{pattern} $P_i = (V_i, l_i)$ is a pair consisting of the solid torus $V_i=S^{1}\times D^{2}$ and an oriented link $l_i$ in $V_i$, such that $l_i$ is not contained in any 3-ball in $V_i$. We say that a pattern $P_i=(V_i,l_i)$ is \emph{trivial} if $l_i$ is a single curve that is isotopic to the core of the solid torus, modulo orientation.

A \emph{(total) pattern} $P$ is an $m$-tuple of patterns $P=(P_1,\ldots,P_m)$ that is non-trivial in the sense that at least one $P_i$ is a non-trivial pattern.

\begin{definition}[Braided pattern]
We say that a pattern $P_i=(V_i,l_i)$ is \emph{braided} if $l_i$ is a closed braid in $V_i$, i.e., $l_i$ is transverse to the meridional disks $\{p\} \times D^{2}$ for all $p \in S^{1}$ and all the intersections are positive. 
We say that a total pattern $P$ is \emph{braided} if all $P_i$ are braided.
\end{definition}

Let $f_i: V_i \rightarrow N(C_i)$ be a homeomorphism that sends the longitude $S^{1} \times \{ \ast \}$ to the longitude of $C_i$. Here the longitude of $C_i$ is a simple closed curve on $\partial N(C_i)$ that is null-homologous in $S^{3} \setminus C_i$.

Let $C_P$ be a link given by 
\[ C_P = \bigsqcup_{i=1}^{m} f_i(l_i) \subset S^{3}\]

We call the torus $f(\partial V_i)$ the \emph{companion torus}.
We say that the link $C_P$ is a \emph{satellite} if the companion tori $f(\partial V_i)$ of a non-trivial patterns $V_i$ are essential.
Similarly, we say that a satellite $C_P$ is \emph{braided satellite} if the pattern is braided.

\begin{definition}[Satellite link]
A link $L$ is a \emph{satellite link} (resp. a \emph{braided satellite link}) if it is the satellite $C_P$ for some companion $C$ and pattern $P$ (resp. braided pattern $P$)
\end{definition}

For a satellite link $L$, a way to write $L$ as a satellite $C_P$ may not be unique. This is why we often distinguish a satellite $C_P$ with a satellite link $L$.

\subsection{Regular form of braided satellites}

As we mentioned in the introduction, one can always get a braid representative of a braided satellite $C_P$.

Let $C = C_1 \cup C_2 \cup \cdots \cup C_m$ be a companion and $P=(P_1,\ldots, P_m)$
be a braided pattern. We view the pattern $P_i$ as the closure of an $n_i$-braid $\beta_i \in B_{n_i}$.

Take a braid representative $\beta_{\rm ext} \in B_{r}$ of $C$. 
We denote by $w(C_i,\beta_{\rm ext})$ the writhe of the $i$-th component $C_i$, with respect to the closed braid diagram $\widehat{\beta_{\rm ext}}$.

Let $\overline{\beta_{\rm ext}}$ be the braid obtained from $\beta_{\rm ext}$ by replacing each strand $t$ of $\beta_{\rm ext}$ with $n_i$ parallel strands, if $t$ corresponds to the $i$-th component $C_i$ of $C$.
We may view $\overline{\beta_{\rm ext}}$ as a braiding of tubes. By inserting the braid $\Delta_{n_i}^{-2w(C_i,\beta_{\rm ext})}\beta_i$ inside suitable tubes (here $\Delta_{n_i}^2$ is the full twist of the $n_i$-strand braid), we get a braid $\beta$ which represents $C_P$. We will write such a braid as
\begin{equation}
\label{eqn:regular-form}
 \beta = \overline{\beta_{\rm ext}} (\Delta_{n_1}^{-2w(C_1,\beta_{\rm ext})}\beta_{1},\ldots, \Delta_{n_m}^{-2w(C_m,\beta_{\rm ext})}\beta_{m}) 
 \end{equation}
 We remark that we need to use $\Delta^{-2w(C_i,\beta_{\rm ext})}\beta_i$, not $\beta_i$ because in the satellite construction, the solid torus $V_i$ is attached so that longitudes $S^{1}\times \{*\}$ is identified with the longitude of $C_i$.

We say that the braid $\beta$ of the form \eqref{eqn:regular-form} a \emph{regular form}. That is, we say that a braid is a regular form if $\beta$ is decomposed as the braiding of tubes $\overline{\beta_{\rm ext}}$ and the braids inside the tubes, par each component $C_i$ of $C$ (See Figure \ref{fig:regular-form})

We call the braid $\beta_{\rm ext}$ and $\overline{\beta_{\rm ext}}$ the \emph{exterior braid} and the \emph{exterior tubular braids}, respectively.

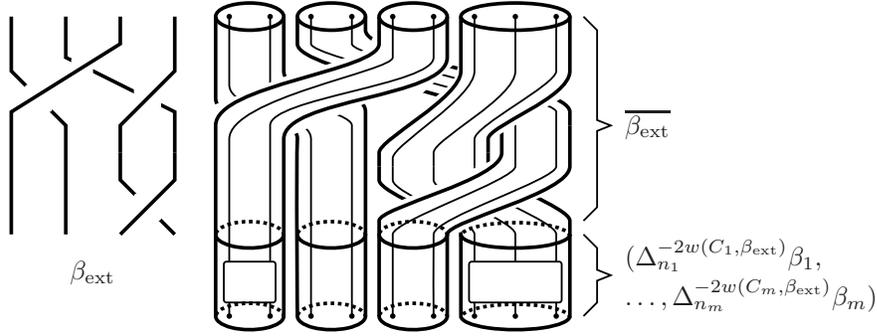
\begin{figure}[htbp]
\begin{center}
\includegraphics*[width=100mm]{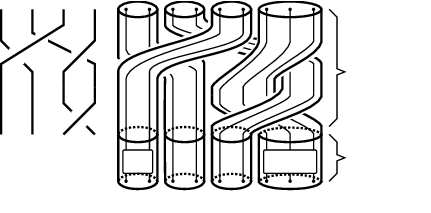}
\begin{picture}(0,0)
\put(-265,35){$\beta_{\rm ext}$}
\put(-55,90){$\overline{\beta_{\rm ext}}$}
\put(-55,40){$ (\Delta_{n_1}^{-2w(C_1,\beta_{\rm ext})}\beta_{1},$}
\put(-55,25){$ \ldots,\Delta_{n_m}^{-2w(C_m,\beta_{\rm ext})}\beta_{m})$}
\end{picture}
\caption{Regular form. The box represents the braid $\Delta_{n_i}^{-2w(C_i,\beta_{\rm ext})}$} 
\label{fig:regular-form}
\end{center}
\end{figure} 


From the construction of regular form, the following is immediate.
{\correction Let $b(C_i,\beta_{\rm ext})$ be the number of the strands of $\beta_{\rm ext}$ that represents the $i$-th component $C_i$.}

\begin{proposition}
\label{proposition:regular-positive}
A regular form \eqref{eqn:regular-form} is a positive braid that contains at least $k$ full twists if and only if $\beta_{\rm ext}$ and $\Delta_{n_i}^{-2w(C_i,\beta_{\rm ext})}\beta_{i}$ are positive braid that contain at least 
{\correction $kb(C_i,\beta_{\rm ext})$} full twist.
\end{proposition}

\subsection{Braided satellite that gives a (fully) positive braid link}
\label{section:proof_b_a}

As an application of the regular form, we prove the easiest implication of our main theorems.

We emphasize that in Proposition \ref{proposition:regular-positive}, the numbers $w(C_i,\beta_{\rm ext})$ {\correction and $b(C_i,\beta_{\rm ext})$} of full-twists to get a positive braid regular form depends on a choice of $\beta_{\rm ext}$, a positive braid representative of $C$. 

To show that $w(C_i,\beta_{\rm ext})$ and {\correction and $b(C_i,\beta_{\rm ext})$}  do not depend on $\beta_{\rm ext}$ for fully positive braids, we use the following.

\begin{lemma}
\label{lemma:sub-link}
Let $L=L_1\cup L_2 \cup \cdots \cup L_m$ be a fully positive braid link represented by a positive $n$-braid $\beta$ that contains at least one full twist.
Let $\beta|_{L_i}$ be the braid diagram obtained from the braid $\beta$ by removing the strands which are not $L_i$.
Then $\beta|_{L_i}$ is a positive braid that contains at least one full twist.
In particular,
\begin{itemize}
\item each component $L_i$ of $L$ is a fully positive braid link.
\item $n=b(L)=b(L_1)+b(L_2)+\cdots + b(L_m)$
\item $\beta|_{L_i}$ is the closed $b(L_i)$-braid diagram that represents $L_i$.
\end{itemize}
\end{lemma}
\begin{proof}
This follows from the following two observations and aforementioned result that the the braid index of the closure of a positive $n$ braid that contains at least one full twist is $n$ \cite{FW}.
\begin{itemize}
\item Removing one strand of the full twist $\Delta^{2}_n \in B_n$ of $n$-strands yields a full twist $\Delta^{2}_{n-1} \in B_{n}$ of $n-1$ strands.
\item Removing strands of a positive braid preserves the property that it is positive.
\end{itemize}
\end{proof}

\begin{proof}[Proof of Theorem \ref{theorem:main} (b) $\Rightarrow$ (a)]
Take a positive braid that contains at least one full twist $\beta_{\rm ext}$ whose closure is $C$ and take a regular form $\beta$ that represents the braided satellite $C_P$.
By Lemma \ref{lemma:sub-link}, $C_i$ is represented by a positive $b(C_i)$-braid hence $2g(C_i)-1=-b(C_i)+w(C_i,\beta_{\rm ext})$ {\correction and $b(C_i,\beta_{\rm ext}) = b(C_i)$}. Hence {\correction $w(C_i,\beta_{\rm ext})+b(C_i,\beta_{\rm ext}) = 2g(C_i)+2b(C_i)-1$} 
so by Proposition \ref{proposition:regular-positive}, $L=C_P$ is a fully positive braid link.
\end{proof}

\section{Reducible braids and regular forms}
\label{Section:reducible}

\subsection{Regular forms of reducible braids}

In this section we discuss a close connection between the dynamics of braids and regular forms.

Let $D_n$ be the $n$-punctured disk $D_n = \{z \in \C \: | \: |z| \leq n+1\} \setminus \{x_1,\ldots,x_n\}$, where the $i$-th puncture $x_i$ is given by $x_i=i$.
A simple closed curve $\mathcal{C}$ of $D_n$ is essential if it is not isotopic to $\partial D_n$ and it encloses more than one puncture points.

\begin{definition}[Round curve]
A simple closed curve $\mathcal{C}$ of $D_n$ is \emph{round} if it is isotopic to the geometric circle whose center lies on the real axis, $\{z \in D_n \: | \: |z-r|\leq R\}$ $(r,R \in \R)$. 
\end{definition}

A \emph{multicurve} $\mathcal{C}=\{\mathcal{C}_1,\ldots, \mathcal{C}_{k}\}$ is a collection of simple closed curves $\mathcal{C}_i$, such that $\mathcal{C}$ and $\mathcal{C}_j$ are not isotopic whenever $i\neq j$. 
We say that a multicurve $\mathcal{C}$ is
\begin{itemize}
\item[--] \emph{essential} if all $\mathcal{C}_i$ are essential.
\item[--] \emph{round} if all $\mathcal{C}_i$ are round.
\end{itemize}

As is well-known, the braid group $B_n$ is identified with the mapping class group of  $D_n$. The standard Artin's generator $\sigma_i$ corresponds to the right-handed half-twist along the straight line segment connecting the $i$-th puncture point $x_i$ and $(i+1)$-st puncture point $x_{i+1}$.

Following the Nielsen-Thurston classification, as an element of mapping class group of $D_n$, an element of the braid group is classified into the three types, \emph{periodic, reducible,} and \emph{pseudo-Anosov} (see \cite{FLP} for details). 
Here we only use the reducible braids.

\begin{definition}[Reducible braids]
A braid $\beta \in B_{n}$ is reducible if $\beta$ is represented by a homeomorphism $f_{\beta}:D_{n} \rightarrow D_n$ that preserves essential multicurve $\mathcal{C}$ of $D_n$.
\end{definition}
In the following, by abuse of notation, we frequently confuse the braid $\beta$ with a homeomorphism $f_{\beta}$ that represents $\beta$ and we often say that `$\beta$ preserves the multicurve $\mathcal{C}$'.

\begin{proposition}
\label{proposition:reducible-regular-form}
If $\beta$ is a reducible braid, then after taking conjugate, $\beta$ is written as a regular form \eqref{eqn:regular-form}.

Furthermore,
\begin{itemize}
\item[(i)] If $\beta \in B_{r}$ fixes an essential round multicurve $\mathcal{C}$, then there is a braid $\alpha \in B_{r}$ that fixes the same multicurve $\mathcal{C}$ such that the conjugate $\alpha^{-1}\beta \alpha$ is a regular form.
\item[(ii)] If $\beta$ is a positive braid that fixes an essential round multicurve $\mathcal{C}$, then we may take such $\alpha$ so that the regular form $\alpha^{-1}\beta \alpha$ is a positive braid.
\end{itemize}
\end{proposition}
\begin{proof}
The detailed construction is discussed in \cite[Section 5.1]{GW1}. Here we give an outline of the construction, for reader's convenience.

Since $\beta$ is reducible, it fixes some essential multicurve $\mathcal{C}$.
Although $\mathcal{C}$ may not be round, but by taking conjugate $\beta^{*}$ we may assume that $\beta^{*}$ fixes an essential round multicurve. Then we can decompose $\beta^{*}$ as a braiding of tubes and braiding inside tubes, a form that is quite similar to a regular form, as we illustrate in Figure \ref{fig:pre-regular-form}. 

\begin{figure}[htbp]
\begin{center}
\includegraphics*[width=50mm]{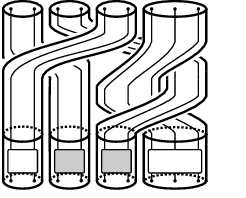}
\begin{picture}(0,0)
\end{picture}
\caption{Reducible braid preserving round essential multicurve is close to a regular form. The difference is that there may be additional braiding inside tubes (gray box)} 
\label{fig:pre-regular-form}
\end{center}
\end{figure} 

By taking further conjugate as depicted in Figure \ref{fig:pre-regular-to-regular}, we gather braidings inside tubes without changing its exterior tubular braids. From the construction, if $\beta^{*}$ is a positive braid this final conjugation preserves the property that $\beta^{*}$ is a positive braid.
\end{proof}

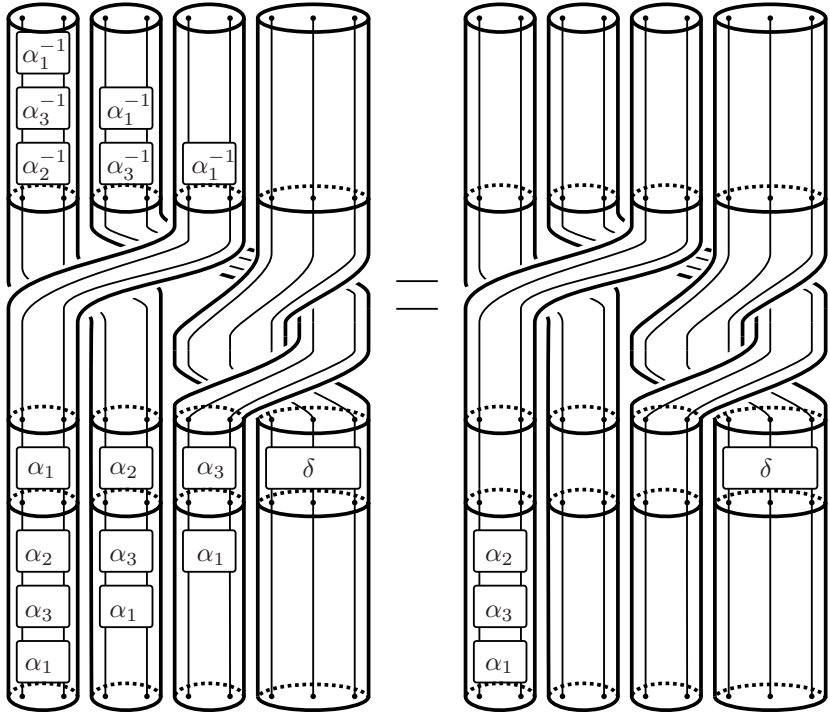
\begin{figure}[htbp]
\begin{center}
\includegraphics*[width=110mm]{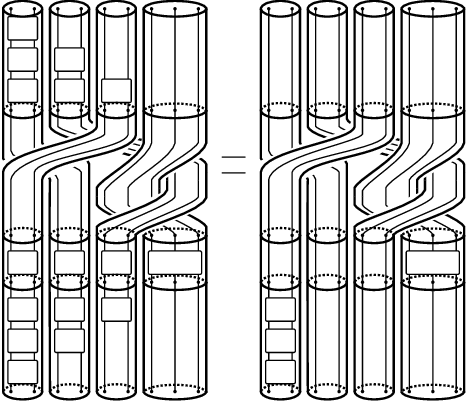}
\begin{picture}(0,0)
\put(-309,247){$\alpha_1^{-1}$}
\put(-309,226){$\alpha_3^{-1}$}
\put(-277,226){$\alpha_1^{-1}$}
\put(-309,205){$\alpha_2^{-1}$}
\put(-277,205){$\alpha_3^{-1}$}
\put(-245,205){$\alpha_1^{-1}$}
\put(-307,91){$\alpha_1$}
\put(-276,91){$\alpha_2$}
\put(-243,91){$\alpha_3$}
\put(-203,90){$\delta$}
\put(-30,90){$\delta$}
\put(-276,58){$\alpha_3$}
\put(-276,37){$\alpha_1$}
\put(-243,58){$\alpha_1$}
\put(-308,58){$\alpha_2$}
\put(-308,37){$\alpha_3$}
\put(-308,17){$\alpha_1$}
\put(-134,58){$\alpha_2$}
\put(-134,37){$\alpha_3$}
\put(-134,17){$\alpha_1$}
\end{picture}
\caption{Taking conjugate to get a regular form} 
\label{fig:pre-regular-to-regular}
\end{center}
\end{figure}

\subsection{Compatible braid representatives}

To relate the reducible braids and satellites, we introduce the following notion.

Let $\beta$ be a reducible braid that fixes a multicurve $\mathcal{C}$, and let $\widehat{\beta}$ be its closure. By taking suspension of the multicurve $\mathcal{C}$ we get a family of torus $T_{\mathcal{C}}=T_1,\ldots,T_m$ of the link complement $S^{3} \setminus \widehat{\beta}$ which we call the \emph{suspension tori} of $\mathcal{C}$.

\begin{definition}
Let $L=C_P$ be a satellite.
We say that a braid representative $\beta$ of $L$ is \emph{compatible} with the satellite $C_P$ if $\beta$ is a reducible braid that fixes a multicurve $\mathcal{C}$ such that the set of suspension tori $T_{\mathcal{C}}=\{T_1,\ldots,T_m\}$ is identical with the set of companion tori $\{f_1(\partial V_1),\ldots, f_{m}(\partial V_m)\}$.
\end{definition}

It is clear that if there is a braid representative that is compatible with the satellite $C_P$, then $C_P$ is a braided satellite. The following gives a geometric characterization of compatibility.

\begin{lemma}
\label{lemma:compatible}
Let $L=C_P$ be a satellite with companion $C$ and pattern $P$.
A braid representative $\beta$ of $L$ is compatible with the satellite $C_P$ if and only if all the companion tori $f_i (\partial V_i)$ can be put so that it is disjoint from the braid axis of $\beta$.
\end{lemma}
\begin{proof}
Only if direction is obvious. 

If we can put all the companion tori $f_i(\partial V_i)$ so that it is disjoint from the braid axis $A$, then they are contained in solid torus $S^{1} \times D^{2} = S^{3} \setminus A$.

Inside the solid torus, we put companion tori so that the number of the connected component of the intersection with $\{0\} \times D^{2}$ is minimum. Then the intersection $\left(\bigcup_{i=1}^{m}f_i(\partial V_i)\right) \cap (\{0\} \times D^{2})$ gives rise to a multicurve of $D_{n}$ preserved by $\beta$, and the companion torus $f_i(\partial V_i)$ is the suspension of the multicurve $f_i(\partial V_i) \cap (\{0\} \times D^{2})$.
\end{proof}

For a braid representative $\beta$ of $L$ compatible with the braided satellite $C_P$, by Proposition \ref{proposition:reducible-regular-form}, there exists a regular form of a braid $\beta^{*}$
\[ \beta^{*} = \overline{\beta^{*}_{\rm ext}} (\Delta_{n_1}^{-2w(C_1,\beta^*_{\rm ext})}\beta_{1},\ldots, \Delta_{n_m}^{-2w(C_m,\beta^*_{\rm ext})}\beta_{m})\]
that is conjugate to $\beta$. Its exterior braid $\beta^*_{\rm ext}$ represents the companion $C$, and the braids $\beta_{1},\ldots,\beta_m$ gives the pattern $P$.

\section{Companion of braided satellite positive braid link}
\label{Section:Garside}

The discussion in the previous section shows for a braid representative $\beta$ compatible with the braided satellite $C_P$, by \emph{taking conjugate}, we may put $\beta$ in a regular form and we get braid representatives of the companion $C$ and the pattern $P$.

In this section, we analyze a conjugation needed to get a regular form by using Garside theory.

\subsection{Garside theory}

First we quickly review a machinery of Garside theory.
For a concise overview of Garside theory, we refer to \cite[Section 1]{BGG}. As a more comprehensive reference we refer to \cite{DDGKM}.

A Garside theory of the braid group $B_{n}$ is a machinery that gives a solution of the word and the conjugacy problems of $B_n$.

Let $\Delta=(\sigma_1\sigma_2\cdots \sigma_{n-1})\cdots (\sigma_1\sigma_{2}) \sigma_1$ be the half-twist braid. A positive braid $x$ is called a \emph{simple braid} if $x^{-1}\Delta \in B_{n}^{+}$.
For an $n$-braid $\beta \in B_{n}$, the \emph{infimum} $\inf(\beta)$ and \emph{supremum} of $\beta$ are defined by
\[ \inf(\beta) = \max\{k \in \Z \: | \: \Delta^{-k}\beta \in B^{+}_{n}\},\  \sup(\beta) = \min\{ k \in \Z \: | \: \beta^{-1}\Delta^{k} \in B^{+}_n\}\]
When $\beta$ is a positive braid, the infimum is the maximum number $k$ such that $\beta$ contains at least $k$ full twists.
  
Every braid $\beta \in B_{n}$ is uniquely represented by the \emph{Garside normal form}
\[ N(\beta) = \Delta^{\inf(\beta)} x_1 x_2 \cdots x_{r}\]
Here $x_i$ are simple braids other than $1,\Delta$.

For a braid $\beta$ with normal form $N(\beta) = \Delta^{\inf(\beta)} x_1 x_2 \cdots x_{r}$
its \emph{cycling} $\mathbf{c}(\beta)$ and \emph{decycling} $\mathbf{d}(\beta)$ are particular conjugates
\begin{align*}
\mathbf{c}(\beta) &= \Delta^{\inf(\beta)}x_2x_3 \cdots x_r(\Delta^{\inf (\beta)}x_1\Delta^{-\inf (\beta)})\\
\mathbf{d}(\beta) &= x_r  \Delta^{\inf(\beta)}x_1x_2 \cdots x_{r-1}.
\end{align*}
The infimum (resp. supermum) never decreases (resp. increases) under the cycling and decycling, namely, $\inf(\beta) \leq \inf(\mathbf{c}(\beta))$ and $\inf(\beta) \leq  \inf(\mathbf{d}(\beta))$ (resp. $\sup(\beta) \geq \sup(\mathbf{c}(\beta))$ and $\sup(\beta) \geq  \sup(\mathbf{d}(\beta))$ hold. 
Let 
\begin{align*}
{\inf}_{s}(\beta)&= \max\{ \inf(\alpha)\: | \: \alpha \in B_{n} \mbox{ is conjugate to } \beta\}\\
{\sup}_{s}(\beta)&= \min\{ \sup(\alpha)\: | \: \alpha \in B_{n} \mbox{ is conjugate to } \beta\}
\end{align*}

The \emph{super summit set} $SSS(\beta)$ of $\beta \in B_n$ is defined by
\[ SSS(\beta) = \{ \alpha \in B_{n} \: | \: \inf(\alpha) = {\inf}_{s}(\alpha), \sup(\alpha)={\sup}_s(\alpha), \alpha \in B_{n} \mbox{ is conjugate to } \beta\}\]
By definition, $\beta$ and $\beta'$ are conjugate if and only if $SSS(\beta)=SSS(\beta')$. One can algorithmically compute $SSS(\beta)$ hence we can determine whether $\beta$ and $\beta'$ are conjugate or not.

The cycling and decycling plays a fundamental role to compute the super summit set.

\begin{proposition}\cite{EM}
\label{proposition:cyc-decyc-SSS}
For any braid $\beta \in B_{n}$, after applying finitely many cycling and decycling, we can get an element of $SSS(\beta)$.
\end{proposition}

\subsection{Reducibility and Garside theory}

In \cite{BNG}, it is shown that  one can determine whether a given braid is reducible or not by Garside theory.

A key observation of \cite{BNG} is that the cycling and the decycling preserves the property that the braid preserves round multicurves.

\begin{proposition}\cite{BNG}
\label{proposition:BNG}
Assume that $\beta \in B_{n}$ preserves an essential round multicurve. Then so do its cycling $\mathbf{c}(\beta)$ and decycling $\mathbf{d}(\beta)$.
\end{proposition}

Proposition \ref{proposition:BNG} and Proposition \ref{proposition:cyc-decyc-SSS} leads to the following.

\begin{theorem}\cite{BNG}
\label{theorem:SSS}
If $\beta$ is reducible, there is an element $\beta^{SSS} \in SSS(\beta)$ that preserves essential round curves.
\end{theorem}

Since one can compute $SSS(\beta)$, and one can check whether a given braid $\beta^{SSS} \in SSS(\beta)$ preserves a round multicurve or not, the theorem gives an algorithm to determine the reducibility of braids.

\begin{remark}
The super summit set method for the conjugacy problem was improved for \emph{ultra summit set} \cite{Ge} and \emph{sliding circuits} \cite{GG}. They have more advantages over the supper summit set. Furthermore, Theorem  \ref{theorem:SSS} is generalized and strengthened for ultra summit set \cite{LL} or sliding circuits \cite{GW2}. 
In \cite{LL} it is shown that under some mild assumptions, if a braid $\beta$ is reducible, then \emph{all} the elements of its ultra summit set preserves essential round curves. Similarly, in \cite{GW2} it is shown that if a braid $\beta$ is reducible, then, without any additional assumptions, \emph{all} the elements of its sliding circuits preserve essential round curves, or, `almost' essential round curves.
\end{remark}

\subsection{Positivity criterion of companion and patterns}

Now we are ready to prove the following theorem, which relates the positivity of braid representative compatible with the satellite $C_P$ and the positivities of companion and patterns.

\begin{theorem}
\label{theorem:compatible-satellite}
Let $\beta$ be a {\correction positive} braid representative of a link $L$ compatible with the satellite $C_P$. Assume that $\beta$ contains at least $k$ full twists $(k \geq 0)$. Then the companion $C$ is represented by a positive braid that contains at least $k$ full twists and the $i$-th pattern $P_i$ is a positive braid.

Furthermore, if $k>0$ then $P_i$ contains at least 
{\correction $2g(C_i)+(k+1)b(C_i)-1$} full twists.
\end{theorem}

\begin{proof}
Since $\beta$ is reducible, by Theorem \ref{theorem:SSS}, there is an element $\beta^{SSS} \in SSS(\beta)$ that preserves round curves.
Since $0\leq k \leq \inf(\beta) \leq \inf_s(\beta) = \inf(\beta^{SSS})$, $\beta^{SSS}$ is a positive braid that contains at least $k$ full twists.
By Proposition \ref{proposition:reducible-regular-form}, by taking further conjugate, there is a positive braid $\beta^{*}$ that is a regular form
\[ \beta^{*} = \overline{\beta^*_{\rm ext}} (\Delta_{n_1}^{-2w(C_1,\beta^*_{\rm ext})}\beta_{1},\ldots, \Delta_{n_m}^{-2w(C_m,\beta^*_{\rm ext})}\beta_{m}) \]
such that $\beta^{*}$ also contains at least $k$ full twists. By Proposition \ref{proposition:regular-positive}, $\beta^*_{\rm ext}$ is a positive braid that contains at least $k$ full twists and that $P_i$ is a positive braid that contains at least 
{\correction $w(C_i,\beta^*_{\rm ext})+kb(C_i,\beta^*_{\rm ext})$} full twists. 

If $k>0$, then by Lemma \ref{lemma:sub-link} the diagram $\beta^*_{\rm ext}|_{C_i}$ is a positive closed $b(C_i)$-braid diagram. Thus $w(C_i,\beta^*_{\rm ext})+1=2g(C_i)+b(C_i)-1$ so $P_i$ contains at least 
{\correction $(2g(C_i)+b(C_i)-1)+kb(C_i)$} 
full twists.

\end{proof}

\section{Existence criterion of compatible braid representatives}
\label{Section:FDTC}

Thanks to Theorem \ref{theorem:compatible-satellite}, to complete the proof the implication (a) $\Rightarrow$ (b) of Theorem \ref{theorem:main},
it remains to show the existence of a (fully positive) braid representative $\beta$ compatible with a satellite $C_P$, for \emph{all} satellite decomposition of $L$.

The \emph{Fractional Dehn twist coefficient} (FDTC, in short) of the braid group $B_{n}$ is a map $c:B_n \rightarrow \R$. Here we do not give precise definitions. For details, we refer to \cite{Ma,HKM,IK}. We will use the following properties of FDTC. 

\begin{proposition}
\label{proposition:FDTC-basics}
The FDTC $c:B_{n}\rightarrow \R$ has the following properties
\begin{itemize}
\item[(i)] $c(\Delta^{2N}\alpha)=N + c(\alpha)$ for all $n \in \Z$ and $\alpha \in B_{n}$.
\item[(ii)] If $\beta$ and $\beta'$ are conjugate, then $c(\beta)=c(\beta')$.
\item[(iii)] $c(\beta^{k})=k c(\beta)$.
\item[(iv)] $c(\alpha \beta)\leq c(\alpha \sigma_i \beta)$ for all $i =1,2,\ldots,n-1$.
\item[(v)] Assume that a braid $\beta = \overline{\beta_{\rm ext}} (\beta_1,\ldots,\beta_m)$ is a regular form. Then $c(\beta) = c(\beta_{\rm ext})$.
\end{itemize}
(Note that properties (i) and (iii) implies that $c(\Delta^{2N})=N$.)
\end{proposition}

The following result from the braid foliation/open book foliation theory says that the FDTC gives an obstruction for essential torus tori to have a non-trivial intersection with the braid axis. We refer to \cite{LM} for basics of braid foliation/open book foliation theory.

\begin{theorem}\cite[Theorem 1.2]{It},\cite[Proposition 7.10]{IK}
\label{theorem:OB}
Let $L$ be the closure of a braid $\beta$ and let $T$ be essential tori of the complement of $L$. If $c(\beta) > 1$, then we can put $T$ so that it is disjoint from the braid axis.
\end{theorem}
%

This leads to the following.

\begin{theorem}[Existence of compatible braid representative]
\label{theorem:compatible-FDTC}
Let $\beta \in B_{n}$ be an $n$-braid and let $L=\widehat{\beta}$ be its closure.
Assume that $c(\beta)>1$. If $L$ is a satellite $C_P$, then the pattern $P$ is braided. Furtheremore, $\beta$ is a reducible braid that is compatible with the braided satellite $C_P$.
\end{theorem}
\begin{proof}
Let $A$ be the axis of the closed braid $L=\widehat{\beta}$. 
Since $c(\beta)>1$, by Theorem \ref{theorem:OB}, every essential torus $T$ of the complement of the closed braid $L=\widehat{\beta}$ can be put so that it is disjoint from the braid axis $A$. Hence by Lemma \ref{lemma:compatible} if $L$ is a satellite $C_P$ then $P$ is a braided pattern and $\beta$ is compatible with the satellite $C_P$.
\end{proof}

Unfortunately, Theorem \ref{theorem:compatible-FDTC} does not cover all the fully positive braid links. For a positive braid that contains at least one full twist, we only say $c(\beta)\geq 1$ by Proposition \ref{proposition:FDTC-basics} (i), (iv).
However, we have the following characterization of a positive braid that contains at least one full twists whose FDTC is one.

\begin{lemma}
\label{lemma:FDTC-1}
Let $\beta=\Delta^{2}\alpha \in B_{n}$ ($\alpha \in B^{+}_{n}$) be a positive $n$-braid that contains at least one full twist.
\begin{itemize}
\item[(i)] $c(\beta)>1$ if and only if $\alpha$ can be written as a positive braid word that contains all $\sigma_1,\ldots, \sigma_{n-1}$.
\item[(ii)] $c(\beta) =1$ if and only if $\beta$ is a regular form whose exterior braid is the full-twist
\[ \beta = \overline{\Delta_{m}^{2}}(\alpha_1,\ldots,\alpha_m)\]
\end{itemize}
\end{lemma}
\begin{proof}
Assume that $\alpha$ can be written as a positive braid word that contains all $\sigma_1,\ldots, \sigma_{n-1}$. Then one can get the braid $\alpha^{n-1}$ from $(\sigma_1\sigma_2\cdots\sigma_{n-1})$ by inserting positive braid $\sigma_i$ repeatedly. By \ref{proposition:FDTC-basics} (iv), this implies $c(\alpha)\geq c(\sigma_1\sigma_2\cdots\sigma_{n-1})$. By \ref{proposition:FDTC-basics} (iii), $c(\sigma_1\sigma_2\cdots\sigma_{n-1})^{n} = c(\Delta^{2})=1$ hence $c(\alpha) \geq \frac{1}{n}>0$ so $c(\beta)=c(\alpha) +1 >1$.

Now we assume that $\alpha$ can be written as a positive braid word that does not contain $\sigma_{i_1},\sigma_{i_2},\ldots,\sigma_{i_m}$ ($i_1<i_2<\cdots <i_m$) but contains all the other generators. Let $\mathcal{C}_{j}$ be the round curve that encloses $p_{i_{j-1}+1}, p_{i_{j-2}+2},\ldots, p_{i_j}$ (here we put $i_{0}=0$). Then $\beta=\Delta^{2}\alpha$ preserves all the round curves $\mathcal{C}_j$, and is a regular form whose exterior braid is the full twist.
\end{proof}

Using this characterization, we complete the proof of existence of compatible braid representatives for fully positive braid links.

\begin{theorem}[Existence of compatible braid representative for fully positive braid link]
\label{theorem:compatible-fully-positive}
Let $\beta \in B_{n}$ be a positive $n$-braid that contains at least one positive full-twist and let $L=\widehat{\beta}$ be its closure. If $L$ is a satellite $C_P$, then the pattern $P$ is braided. Furthermore, $\beta$ is a reducible braid that is compatible with the satellite $C_P$.
\end{theorem}
\begin{proof}
Let $L$ be a fully positive braid link.
Take a braid representative $\beta=\Delta^{2}\alpha \in B_{n}$ ($\alpha \in B^{+}_{n}$ of $L$ so that it at least one full twist.

If $c(\beta)>1$, the assertion follows from Theorem \ref{theorem:compatible-FDTC}, so we assume that $c(\beta)=1$. Thus by Lemma \ref{lemma:FDTC-1}, $L$ is a braided satellite of $(m,m)$-torus link. From the JSJ decomposition of link complements, it follows that every essential torus $T$ in the complement of the closed braid $L=\widehat{\beta}$ is disjoint from the braid axis $A$. Thus By Lemma \ref{lemma:compatible}, if we write $L$ as a satellite $L=C_P$ then the braid $\beta$ is compatible with the satellite $C_P$.
\end{proof}

This completes the proof of our main theorem.

\begin{proof}[Proof of Theorem \ref{theorem:main} (a) $\Rightarrow (b)$]
Assume that a satellite $C_P$ is a fully positive braid link represented by a positive braid $\beta$ that contains at least one full twist.
 By Theorem \ref{theorem:compatible-fully-positive}, $\beta$ is compatible with the satellite $C_P$. By Theorem \ref{theorem:compatible-satellite}, the companion $C$ is a fully positive braid link, and the $i$-th pattern $P_i$ is a positive braid that contains at least 
 {\correction $2g(C_i)+2b(C_i)-1$} full twists.
\end{proof}

The same argument shows the following variant of our main theorem. 

\begin{theorem}[Characterization of satellite positive braid link under the FDTC$>1$ condition]
\label{theorem:main-variant}
Let $L$ be the satellite $C_P$ with companion $C=C_1 \cup C_2\cup \cdots \cup C_{m}$ and pattern $P=(P_1,\ldots,P_m)$.
Then (a) and (b) are equivalent.
\begin{itemize}
\item[(a)] $L$ admits a positive braid representative $\beta$ such that $c(\beta)>1$
\item[(b)] $C$ is a positive braid link that admits a positive braid representative $\beta_{\rm ext}$ such that $c(\beta_{\rm ext})>1$ and the pattern $P$ is braided. The $i$-th pattern $P_i$ is a positive braid that contains at least $w(C_i,\beta_{\rm ext})$
 full twists. 
\end{itemize}
\end{theorem}
\begin{proof}
The proof is almost the same. (b) $\Rightarrow$ (a) is easy. To see (a) $\Rightarrow$ (b), let $\beta$ be a positive braid representative of the link $L$ with $c(\beta)>1$. By Theorem \ref{theorem:compatible-FDTC}, $\beta$ is compatible with the satellite $C_P$. By Theorem \ref{theorem:compatible-satellite}, the companion $C$ is a positive braid link represented by $\beta_{\rm ext}$, and the $i$-th pattern $P$ is a positive braid that contains at least $w(C_i,\beta_{\rm ext})$ full twists, where $\beta_{\rm ext}$ is an exterior braid of a positive regular form of $\beta$. 
By Proposition \ref{proposition:FDTC-basics} (v), $c(\beta_{\rm ext})=c(\beta)>1$. 
\end{proof}

There are many positive braids $\beta$ that contain no full twists but $c(\beta)>1$. For example, a $3$-braid $\beta=\sigma_1^{a}\sigma_2^{b}\sigma_1^{c}\sigma_2^{d}\sigma_{1}^{e}\sigma_{2}^{f}$ for $a,b,c,d,e,f\geq 2$) is positive, $c(\beta)>1$ but does not contain a full twist. 

\setcounter{section}{1}
\setcounter{theorem}{0}
\renewcommand{\thesection}{\Alph{section}}
\section*{Appendix: characterization of the unknot by positive braid properties}
\label{Section:appendix}

In Corollary \ref{corollary:char-unknot} we gave a characterization of the unknot by the property that certain braided satellite is a fully positive braid knot.

In this Appendix, we give a similar characterization that generalizes the following.

\begin{theorem}\cite{Kr}
\label{theorem:K}
The $(n,\pm 1)$ cable $K_n$ of a knot $K$ is a positive braid knot for some $n>1$ if and only if $K$ is the unknot.
\end{theorem}
Her proof uses deep machineries like Gordon-Luecke theorem \cite{GL} and based on a result on existence of taut foliations.

We generalize Theorem \ref{theorem:K} for more general braided satellites.
We view a pattern $P=(V,l)$ as a link in $S^{3}$ by taking its standard embedding $V \hookrightarrow S^{3}$, in other words, by taking the satellite $U_P$ with the companion the unknot $U$. 

\begin{theorem}\label{theorem:var-K}
Let $P$ be a braided pattern represented by a positive $n$-braid. Assume that $p(P)\neq 1$ (i.e. $P$ represents either the unknot, or, a non-prime knot).
Then the satellite $K_P$ is a positive braid knot if and only if $K$ is the unknot.
\end{theorem}

Our proof uses the property of Alexander polynomial of positive braid knot. Let
\[ \Delta_K(t)=\pm 1 + \sum_{i=1}^{N} a_i(K)(t^{i}+t^{-i})\] 
be the Alexander polynomial of a knot $K$, normalized so that $\Delta_K(t^{-1})=\Delta_K(t)$ and that $\Delta_K(1)=1$ hold. 
Let $p(K)$ be the number of prime factors of $K$, defined by
\[ p(K) = \begin{cases} 
\max\{n \: | \: K=K_1\# K_2 \# \cdots \# K_n, K_i \mbox{ is not the unknot}\} & K \neq \mathsf{Unknot}\\
0 & K = \mathsf{Unknot}
\end{cases}
\]

The following proposition, though interesting, is just a restatement of the corresponding result for the Conway polynomial \cite{It-HOMFLY}. This says that the Alexander polynomial of a positive braid knot detects the number of prime factors.

\begin{proposition}
\label{proposition:Alexander}
Let $K$ be a positive braid knot of genus $g=g(K)$. Then $g$ is equal to the degree of the Alexander polynomial, $a_g(K)=1$ and $a_{g-1}(K)=-p(K)$.
\end{proposition}
\begin{proof}
By \cite[Corollary 1.2]{It-HOMFLY} for the Conway polynomial $\nabla_K(z)$ of a positive braid knot $K$, $\nabla_K(z)=z^{2g}+ (2g-p(K))z^{2g-2} + \cdots$. The assertion follows from the fact that $\Delta_K(t)=\nabla_K(t^{1/2}-t^{-1/2})$.
\end{proof}

\begin{proof}[Proof of Theorem \ref{theorem:var-K}]
`If' direction is obvious so we prove `only if'.

Let $\Delta_K(t)$ and $\Delta_P(t)$ be the Alexander polynomial of $K$ and $P$, respectively. By Proposition \ref{proposition:Alexander} 
\[ \Delta_P(t)=(t^{g(P)}-t^{-g(P)}) -p(P) (t^{g(P)-1}+t^{-g(P)+1}) +\cdots \] 
We put 
\[ \Delta_K(t) = a_N(K)(t^{N}+t^{-N}) + a_{N-1}(K)(t^{N-1}+t^{-N+1})+ \cdots \]
By the satellite formula of the Alexander polynomial $\Delta_{K_P}(t)= \Delta_K(t^n)\Delta_P(t)$, it follows that
\[ \Delta_{K_P}(t)=a_{N}(K)(t^{nN+g(P)}-t^{-nN-g(P)}) -p(P)(t^{nN+g(P)-1}-t^{-nN-g(P)+1}) + \cdots \]
If $K_P$ is a positive braid knot, by Proposition \ref{proposition:Alexander}  $nN+g(P)=g(K_P)= ng(K)+g(P)$ hence $N=g(K)$. Furthermore, $p(P)=p(K_{P})$.
Since $P(K_P)=1$ whenever $K$ is non-trivial and we are assuming $p(P)\neq 1$, this shows that $K$ must be the unknot. 
\end{proof}


\begin{thebibliography}{1}

\bibitem[BGG]{BGG}
J.\ Birman, V.\ Gebhardt, and J.\ Gonz\'alez-Meneses, J
{\em Conjugacy in Garside groups. I. Cyclings, powers and rigidity.}
Groups Geom. Dyn. 1(2007), no.3, 221--279.

\bibitem[BK]{BK}
J.\ Birman and I.\ Kofman, 
{\em A new twist on Lorenz links.}
J. Topol.2(2009), no.2, 227--248.

\bibitem[BW]{BW}
J.\ Birman and R.\ F.\ Williams, 
{\em Knotted periodic orbits in dynamical systems. I. Lorenz's equations.}
Topology 22(1983), no.1, 47--82.

\bibitem[BNG]{BNG}
D.\ Bernardete, Z.\ Nitecki, and M.\ Guti\'errez, 
{\em Braids and the Nielsen-Thurston classification.}
J. Knot Theory Ramifications 4(1995), no.4, 549--618.

\bibitem[DDGKM]{DDGKM}
P.\ Dehornoy, F.\ Digne, E.\ Godelle, D.\ Krammer, and J.\ Michel, 
{\em Foundations of Garside theory.}
EMS Tracts Math., 22
European Mathematical Society (EMS), Z\"urich, 2015. xviii+691 pp.

\bibitem[EM]{EM}
A. El-Rifai, and H.\ Morton,
{\em Algorithms for positive braids.}
Quart. J. Math. Oxford Ser. (2)45(1994), no.180, 479--497.

\bibitem[FLP]{FLP}
A.\ Fathi, F.\ Laudenbach, and V.\ Po\'enaru,
{\em Thurston's Work on Surfaces} Translated from the 1979 French original by D. Kim and D. Margalit. Math. Notes, 48. Princeton University Press, Princeton, NJ, 2012. xvi+254 pp.

\bibitem[FW]{FW}
J.\ Franks and R.\ Williams,
{\em Braids and the Jones polynomial.}
Trans. Amer. Math. Soc. 303(1987), no.1, 97--108.

\bibitem[Ge]{Ge}
V.\ Gebhardt, 
{\em A new approach to the conjugacy problem in Garside groups.}
J. Algebra 292(2005), no.1, 282--302.

\bibitem[GG]{GG}
V.\ Gebhardt and J.\ Gonz\'alez-Meneses,
{\em The cyclic sliding operation in Garside groups.}
Math. Z.265(2010), no.1, 85--114.

\bibitem[GW1]{GW1}
J.\ Gonz\'alez-Meneses and B.\ Wiest,
{\em On the structure of the centralizer of a braid.}
Ann. Sci. \'Ecole Norm. Sup. (4)37(2004), no.5, 729--757.

\bibitem[GW2]{GW2}
J.\ Gonz\'alez-Meneses and B.\ Wiest,
{\em Reducible braids and Garside theory.}
Algebr. Geom. Topol. 11(2011), no.5, 2971--3010.

\bibitem{GL}
C.\ Gordon and J.\ Luecke,
{\em Knots are determined by their complements.}
J. Amer. Math. Soc.2(1989), no.2, 371--415.

\bibitem[HKM]{HKM}
K.\ Honda, W.\ Kazez, and G.\ Mati\'c,
{\em Right-veering diffeomorphisms of compact surfaces with boundary.}
Invent. Math. 169(2), 427--449 (2007) 


\bibitem[It1]{It}
T.\ Ito,
{\em Braid ordering and the geometry of closed braid,}
Geom. Topol. 15 (2011), 473--498.

\bibitem[It2]{It-HOMFLY}
T.\ Ito,
{\em A note on HOMFLY polynomial of positive braid links.}
Internat. J. Math.33(2022), no.4, Paper No. 2250031, 18 pp.

\bibitem[IK]{IK}
T.\ Ito and K.\ Kawamuro,
{\em Essential open book foliations and fractional Dehn twist coefficient.}
Geom. Dedicata 187(2017), 17--67.

\bibitem[Kr]{Kr}
S.\ Krishna,
{\em Taut foliations, braid positivity, and unknot detection.}
arXiv:2312.00196.

\bibitem[KM]{KM}
S.\ Krishna and H.\ Morton,
{\em Twist positivity, L-space knots, and concordance,}
Selecta Math. (N.S.) 31 (2025), no. 1, Paper No. 11, 27 pp.

\bibitem[LM]{LM}
D.\ LaFountain, and W.\ Menasco, 
{\em Braid foliations in low-dimensional topology.}
Grad. Stud. Math., 185
American Mathematical Society, Providence, RI, 2017. xi+289 pp.

\bibitem[LL]{LL}
E-K.\ Lee and S-J.\ Lee,
{\em A Garside-theoretic approach to the reducibility problem in braid groups.}J. Algebra 320(2008), no.2, 783--820.

\bibitem[Ma]{Ma}
A. Malyutin,
{\em Writhe of (closed) braids,}
Algebra i Analiz 16 (2004) 59--91. translated in St. Petersburg Math. J. 16 (2005) 791--813.

\bibitem[dPP]{dPP}
T.\ de Paiva and J.\ Purcell,
{\em Satellites and Lorenz knots.}
Int. Math. Res. Not. IMRN(2023), no.19, 16540--16573.

\bibitem[St]{St}
A.\ Stoimenow, 
{\em On polynomials and surfaces of variously positive links.}
J. Eur. Math. Soc. (JEMS)7(2005), no.4, 477--509.

\bibitem[vB]{vB}
J.\ M.\ Van Buskirk, 
{\em Positive knots have positive Conway polynomials.}
Knot theory and manifolds (Vancouver, B.C., 1983), 146--159.
Lecture Notes in Math., 1144
\end{thebibliography}
\end{document}